\documentclass{article}

\usepackage{amsfonts}
\usepackage{amssymb}
\usepackage{theorem}
\usepackage{amsmath}
\usepackage{latexsym}
\usepackage[latin1]{inputenc}
\usepackage{amsbsy}
\usepackage{enumitem}
\usepackage[colorlinks=true]{hyperref}
\hypersetup{urlcolor=blue, citecolor=red}

\allowdisplaybreaks[1]

\numberwithin{equation}{section}
{\theoremstyle{break}\newtheorem{theorem}{Theorem}}

{\theoremstyle{break}\theorembodyfont{\rmfamily}}

{\theoremstyle{break}\theorembodyfont{\rmfamily}\newtheorem{definition}[theorem]{Definition}}

{\theoremstyle{break}\newtheorem{lemma}[theorem]{Lemma}}
{\theoremstyle{break}\newtheorem{proposition}[theorem]{Proposition}}
{\theoremstyle{break}\newtheorem{corollary}[theorem]{Corollary}}

{\theoremstyle{break}\theorembodyfont{\rmfamily}}

{\theoremstyle{break}\theorembodyfont{\rmfamily}\newtheorem{remark}[theorem]{Remark}}

\newenvironment{proof}[1][Proof]{\noindent\textbf{#1.} }{\
\rule{0.5em}{0.5em}}

\newcommand{\R}{\mathbb{R}}

\newcommand{\N}{\mathbb{N}}

\newcommand{\TT}[1]{\mathcal{T}_{#1}}
\newcommand{\Tt}[1]{\mathcal{T}^1_{#1}}

\newcommand{\card}[1]{\operatorname*{card}({#1})}
\newcommand{\sgn}[1]{\operatorname*{sgn}({#1})}

\begin{document}
\author{}
\date{}
\title{\scshape Notions of Affinity in Calculus of Variations with Differential Forms}
\maketitle

\centerline{\scshape Saugata Bandyopadhyay }
\medskip
{\footnotesize
\centerline{Department of Mathematics \& Statistics}
 \centerline{IISER Kolkata}
   \centerline{Mohanpur-741246, India}
   \centerline{saugata.bandyopadhyay@iiserkol.ac.in}

}

\medskip

\centerline{\scshape Swarnendu Sil}
\medskip
{\footnotesize
 \centerline{ Section de Math\'{e}matiques}
   \centerline{Station 8, EPFL}
   \centerline{1015 Lausanne, Switzerland}
   \centerline{swarnendu.sil@epfl.ch}
}

\begin{abstract}
\noindent
Ext-int.\ one affine functions are functions affine in the direction of one-divisible exterior forms, with respect to exterior product in one variable and with respect to
interior product in the other. The purpose of this article is to prove a characterization theorem for this class of functions, which plays an important role
in the calculus of variations for differential forms.
\end{abstract}

\section{Introduction}
In this article, we introduce the notion of ext-int.\ one convex functions and study the structure of its affine analogue. This class of functions arise naturally in the context of calculus of variations
when we consider the minimization problem for integrals of the form
$$
I(\omega)=\int_{\Omega}f\left(d\omega,\delta\omega \right),
$$
where $1\leqslant k\leqslant n-1$, $f:\Lambda^{k+1}\times\Lambda^{k-1}\rightarrow\mathbb{R}$ is continuous, $\Omega\subseteq\R^n$ is open,
$\omega:\Omega\rightarrow\Lambda^k$ is a $k$-form and $d$, $\delta$ are exterior derivative and codifferential operators respectively. In particular, when $k=1$, identifying one-forms with vector fields, the minimization problem can be seen as the minimization of one involving the curl and the divergence, see Barroso-Matias \cite{barroso-matias}, Dacorogna-Fonseca \cite{dac-fonseca} and references therein. A subclass of the class of functions aforementioned above, namely the class of ext.\ one convex functions, was first introduced in Bandyopadhyay-Dacorogna-Sil \cite{bds} to handle minimization problems where $f$ depends only on the exterior derivative. In the same article, a characterization theorem was
obtained for ext.\ one affine functions, see Theorem 3.3 of \cite{bds}.
To extend the framework to the case where $f$ has explicit dependence on the codifferential as well, one needs to introduce the notion of ext-int.\ one convex functions which play a role as crucial as that of ext.\ one convex functions in the aforementioned context.
\\ \\
The main goal of this article is to prove a characterization theorem for ext-int.\ one affine functions, see Theorem \ref{Thm principal ext-int. quasiaffine}. In the process, we also find a new proof of the theorem (cf.\ Theorem \ref{characterizations of affine functions}) that characterizes ext.\ one affine functions. The new proof is more algebraic in spirit, constructive through a recursion and provides a different perspective on the result. Additionally, the technique we employ here to handle order-preserving permutations of multiple number of ordered multi-indices in the
course of the proof is of independent value and implicitly already played an important role in connecting the calculus of variations with
forms with the classical vectorial calculus of variations, see Bandyopadhyay-Sil \cite{bs}.
\\ \\
The rest of the article is organized as follows. In Section \ref{notation}, we collect the notations we have used throughout the article. Section \ref{Notions of exterio-interior convexity} introduces various classes of exterior convex functions i.e. functions that are convex with respect to the exterior structure.  A few algebraic lemmas  are proved in Section \ref{algebraiclemmas}, which are used in Section \ref{main} to prove the main theorem that reads as follows
\begin{theorem}
Let $1\leqslant k\leqslant n-1$. Then, $f:\Lambda^{k+1}\times \Lambda^{k-1}\rightarrow\mathbb{R}$ is ext-int.\ one affine
if and only if there exist $c_{s}\in\Lambda^{(k+1)s}$ and $d_{r}\in\Lambda^{(n-k+1)r},$ for all $0\leqslant s\leqslant\left[
\frac{n}{k+1}\right]  $,\,$0 \leqslant r\leqslant\left[
\frac{n}{n-k+1}\right] $ such that,
\[
f\left(  \xi , \eta \right)  =\sum_{s=0}^{\left[ \frac{n}{k+1}\right]  }\left\langle c_{s}%
;\xi^{s}\right\rangle  + \sum_{r=0}^{\left[  \frac{n}{n-k+1} \right]  }\left\langle d_{r}%
;(\ast\eta)^{r}\right\rangle,\text{ for all }\xi\in\Lambda^{k+1},\eta\in\Lambda^{k-1}.
\]
\end{theorem}
The aforementioned theorem has a curious implication. Note that, we have nonlinearity in $\xi$ and $\eta$ if and only if $k$ is odd, $n$ is even,  $n\geqslant 2(k+1)$ and $n\geqslant 2(n-k+1)$.
Since the two inequalities are never satisfied simultaneously, we have nonlinearity at the most in one variable, the other variable
appearing as an affine term, see Corollary \ref{corollary to the main theorem}. This observation is an important one in the context of calculus of variations involving differential forms, as it controls the way a variational problem behaves as a function of the order of the form. In spite of being a problem of vectorial nature, the variational problem always behaves as though it is a scalar one with respect to one of the variables.
\section{Notations}\label{notation}
Let $n\in\N$ and let $k\in\N\cup\{0\}$.
\begin{enumerate}[leftmargin=*]
\item $\Lambda^{k}\left(  \mathbb{R}^{n}\right)$ (or simply
$\Lambda^{k}$) denotes the vector space of all alternating $k$-linear maps. For $k=0,$ we set $\Lambda^{0}=\mathbb{R}$. Note that, $\Lambda^{k}\left(  \mathbb{R}%
^{n}\right)  =\{0\}$ for $k>n$, and, for $k\leq n,$ $\operatorname{dim}\left(
\Lambda^{k}\left(  \mathbb{R}^{n}\right)  \right)  ={\binom{{n}}{{k}}}.$
\item For $1\leqslant k\leqslant n$, we write $\mathcal{T}_k:=\{(i_1,\ldots,i_k)\in\N^k:1\leqslant i_1<\cdots<i_k\leqslant n\}$ and for each $r\in\{1,\ldots,n\}$, $\mathcal{T}^r_k:=\{I\in\mathcal{T}_k:r\notin I\}$. Let $I\in\mathcal{T}_k$ and let $I:=(i_1,\ldots,i_k)$. For each $1\leqslant p\leqslant k$, we write $I(i_p):=(i_1,\ldots,i_{\hat{p}},\ldots,i_k)$, where $\hat{p}$ denotes the absence of the index $p$. Note that, $I(i_p)\in\mathcal{T}_{k-1}$, for all $1\leqslant p\leqslant k$.
\item $\wedge,$ $\lrcorner\,,$ $\left\langle ;\right\rangle $ and $\ast$ denote the exterior product, the interior product, the
scalar product and the Hodge star operator respectively.
\item We use the multi-index notation often. For $I=(i_1,\ldots,i_k)\in \mathcal{T}_k$, we write $e^I$ to denote $e^{i_{1}}\wedge\cdots\wedge e^{i_{k}}$. In this notation, if $\left\{  e^{1},\cdots,e^{n}\right\}  $ is a basis of $\mathbb{R}^{n},$ then, identifying $\Lambda^{1}$ with $\mathbb{R}^{n}$, it follows that $\left\{e^I:I\in\mathcal{T}_k\right\}$ is a basis of $\Lambda^{k}.$
\item Let $\omega\in\Lambda^k$ and let $0\leqslant s\leqslant k\leqslant n$. The space of interior annihilators of $f$ of order $s$ is
$$
\text{\upshape{Ann}}_{\lrcorner}\left(\omega,s\right):=\{f\in\Lambda^s:f\,\lrcorner\,\omega=0\}.
$$
Furthermore, we define the rank of order $s$ of $\omega$ as
$$
\text{\upshape{rank}}_s(\omega):=\binom{n}{s}-\operatorname*{dim}\left(\text{\upshape{Ann}}_{\lrcorner}\left(\omega,s\right)\right).
$$
See \cite{CDK}, \cite{DK} for more details on rank and annihilator.
\item Let $m,n\in\N$ and let $r_1,\ldots,r_m\in\N$ be such that $r_1+\cdots+r_m\leqslant n$. For all $j=1,\ldots,m$, let $I_j\in\mathcal{T}_{r_j}$ satisfy $I_p\cap I_q=\emptyset$, for all $p\neq q$. Then, we define $\left[I_1,\ldots,I_m\right]$ to be the permutation of $\left(I_1,\ldots,I_m\right)$ such that $$\left[I_1,\ldots,I_m\right]\in \mathcal{T}_{r_1+\cdots+r_m}.$$
Furthermore, we define the sign of $\left[I_1,\ldots,I_m\right]$, denoted by $\operatorname*{sgn}\left(I_1,\ldots,I_m\right)$, as
$$
e^{\left[I_1,\ldots,I_m\right]}:=\operatorname*{sgn}\left(I_1,\ldots,I_m\right)e^{I_1}\wedge\cdots\wedge e^{I_m}.
$$
\end{enumerate}

\noindent Concerning the last notation, the following properties are easy to check, which we record for the sake of completeness.
\begin{proposition}
Let $m,n\in\N$ and let $r_1,\ldots,r_m\in\N$ be such that $r_1+\cdots+r_m\leqslant n$. For all $j=1,\ldots,m$, let $I_j\in\mathcal{T}_{r_j}$ satisfy $I_p\cap I_q=\emptyset$, for all $p\neq q$. Then,
\begin{itemize}
\item [$(i)$]$[I_1,I_2]=[I_2,I_1]$ and $[I_1,I_2,I_3]=\left[I_1,[I_2,I_3]\right]=\left[[I_1,I_2],I_3\right]$.
\item [$(ii)$]$\sgn{I_1,I_2}=(-1)^{r_1r_2}\sgn{I_2,I_1}$.
\item [$(iii)$]$\sgn{I_1,I_2,I_3}=\sgn{I_2,I_3}\sgn{I_1,[I_2,I_3]}=\sgn{I_1,I_2}\sgn{[I_1,I_2],I_3}.$
\item [$(iv)$]If $I\in\TT{k}$ is written as $I:=(i_1,\ldots,i_k)$, for all $\mu,\nu=1,\ldots,k$,
$$\sgn{I(i_\mu),i_\mu}=(-1)^{k-\mu}\text{ and }\sgn{i_\nu,I(i_\nu)}=(-1)^{\nu-1}.$$
\item [$(v)$]For all $\omega\in\Lambda^k, \varphi\in\Lambda^l$ and $I\in \TT{k+l}$,
$$
\left\langle \omega\wedge\varphi; e^I\right\rangle =\sum_{\substack{R\in\mathcal{T}_{k},S\in\mathcal{T}_{l},\\R\cup S=I,R\cap S=\emptyset}}\operatorname*{sgn}(R,S)\left\langle\omega;e^R\right\rangle\left\langle\varphi;  e^S\right\rangle.
$$
\end{itemize}
\end{proposition}
\section{Notions of exterior convexity}\label{Notions of exterio-interior convexity}
Let us introduce the following classes of functions convex with respect to the exterior structure. We will restrict ourselves to the corresponding affine classes in the subsequent sections.
\begin{definition}\label{29.6.118}
Let $1\leqslant k\leqslant n-1$ and let $f:\Lambda^{k+1}\times\Lambda^{k-1}\rightarrow\R$. We say that
\begin{enumerate}[leftmargin=*]
\item $f$ is \emph{ext-int.\ one convex} if for every $\xi\in \Lambda^{k+1}$, $\eta\in \Lambda^{k-1}$, $a\in \Lambda^1$ and $b\in\Lambda^{k}$, the function $g:\R\rightarrow\R$ defined as
$$
g(t):=f(\xi+t\,a\wedge b,\eta+t\,a\lrcorner b),\text{ for all }t\in\R,
$$
is convex. Furthermore, $f$ is said to be \emph{ext-int.\ one affine} if $f$, $-f$ are both ext-int.\ one convex.
\item $f$ is \emph{ext-int.\ quasiconvex} if $f$ is locally integrable, Borel measurable and
\[
\int_{\Omega}f\left(  \xi+d\omega, \eta + \delta \omega \right)\geqslant f\left(  \xi, \eta \right)
\operatorname*{meas}\Omega,
\]
for every open, bounded set $\Omega\subseteq\R^n$, $\xi\in\Lambda^{k+1}$, $\eta\in \Lambda^{k-1} $ and $\omega\in W_{0}^{1,\infty}\left(  \Omega;\Lambda^{k}\right)$. Moreover, $f$ is said to be \emph{ext-int.\ quasiaffine} if $f$, $-f$ are both ext-int.\ quasiconvex.
\item $f$ is \emph{ext-int.\ polyconvex} if there exists a convex function
\[
F:\Lambda^{k+1}\times\cdots\times\Lambda^{(k+1)\left[ \frac{n}{k+1}\right]
}\times \Lambda^{n-k+1}\times  \cdots \times \Lambda^{(n-k+1)\left[ \frac{n}{n-k+1}\right]}\rightarrow\mathbb{R}%
\]
such that, for all $\xi\in \Lambda^{k+1}$, $\eta\in \Lambda^{k-1}$,
\[
f\left(  \xi, \eta \right)  =F\left(  \xi,\ldots,\xi^{\left[ \frac{n}{k+1}\right]
}, \ast \eta, \ldots ,(\ast\eta)^{\left[  \frac{n}{n-k+1}\right]} \right)  .
\]
Furthermore, $f$ is said to be \emph{ext-int.\ polyaffine} if $f$, $-f$ are both ext-int.\ polyconvex.
\end{enumerate}
\end{definition}
Recall that the following classes were introduced in \cite{bds}.
\begin{definition}\label{29.6.1}
\begin{enumerate}[leftmargin=*]
\item Let $1\leqslant k\leqslant n$. We say that  $f:\Lambda^k\rightarrow\R$ is \emph{ext.\ one convex} if for every $\omega\in \Lambda^k$, $a\in \Lambda^1$ and $b\in\Lambda^{k-1}$, the function $g:\R\rightarrow\R$ defined as
$$
g(t):=f(\omega+t\,a\wedge b),\text{ for all }t\in\R,
$$
is convex. Furthermore, $f$ is said to be \emph{ext.\ one affine} if $f$, $-f$ are both ext.\ one convex.
\item Let $0\leqslant k\leqslant n-1$. We say that $f:\Lambda^k\rightarrow\R$ is \emph{int.\ one convex} if for every $\omega\in \Lambda^k$, $a\in \Lambda^1$ and $b\in\Lambda^{k+1}$, the function $g:\R\rightarrow\R$ defined as
$$
g(t):=f(\omega+t\,a\lrcorner b),\text{ for all }t\in\R,
$$
is convex. Furthermore, $f$ is said to be \emph{int.\ one affine} if $f$, $-f$ are both int.\ one convex.
\end{enumerate}
\end{definition}
The notion of Hodge transform allows us to go back and forth between ext.\ one convex and int.\ one convex functions, see Remark \ref{23.4.2014.1}.
\begin{definition}[Hodge transform]
Let $0\leqslant k\leqslant n$ and let $f:\Lambda^k\rightarrow\R$. The Hodge transform of $f$ is the function $f_\ast:\Lambda^{n-k}\rightarrow\R$ defined as
$$
f_\ast(\omega):=f\left(\ast\omega\right),\text{ for all }\omega\in \Lambda^{n-k}.
$$
\end{definition}
\begin{remark}\label{23.4.2014.1}
\begin{enumerate}[leftmargin=*]
\item Evidently, every convex function is ext-int.\ polyconvex. Furthermore, using standard techniques of calculus of variations, we have the following chain of implications
\[
\text{ ext-int.\ polyconvexity}\Rightarrow\text{ ext-int.\ quasiconvexity}\Rightarrow\text{ ext-int.\ one convexity}.
\]
\item Ext-int.\ polyconvexity is equivalent to convexity when both of $k,n$ are even, or when $n\in\{2k-1,2k,2k+1\}$.
\item The duality between the aforementioned notions of convexity is reflected through the following observation. When $0\leqslant k\leqslant n-1$, $f$ is int.\ one convex if and only if $f_\ast$ is ext.\ one convex. Similarly, when $1\leqslant k\leqslant n$, $f$ is ext.\ one convex if and only if $f_\ast$ is int.\ one convex.
\item When $k=1,n-1,n$, or $k=n-2$ with $n$ odd, ext.\ one convexity is equivalent to convexity. See \cite{bds} for more details on ext.\ one convex functions.
\end{enumerate}
\end{remark}
The following lemma relates ext-int.\ one convexity with ext.\ one and int.\ one convexity.
\begin{lemma}\label{corollary of crucial lemma for extintaffine}
Let $1\leqslant k\leqslant n-1$ and let $f:\Lambda^{k+1}\times \Lambda^{k-1}\rightarrow\mathbb{R}$ be ext-int.\ one convex (resp.\ ext-int.\ one affine). Then, the following holds true
\begin{enumerate}
 \item[$(i)$] The function $f_\eta:\Lambda^{k+1}\rightarrow\R$ defined as
 $$f_\eta(\xi) := f(\xi,\eta),\text{ for all }\xi\in\Lambda^{k+1}$$
 is ext.\ one convex (resp.\ ext.\ one affine), for every $\eta \in \Lambda^{k-1}$.
 \item[$(ii)$] The function $f^\xi:\Lambda^{k-1}\rightarrow\R$ defined as
 $$f^\xi(\eta) := f(\xi,\eta),\text{ for all }\eta\in\Lambda^{k-1}$$
 is int.\ one convex (resp.\ int.\ one affine), for every $\xi \in \Lambda^{k+1}$.
\end{enumerate}
\end{lemma}
\begin{remark}
 The converse of Lemma \ref{corollary of crucial lemma for extintaffine} is false. This can be seen by considering the function $f:\Lambda^2\times\Lambda^0=\Lambda^2\times\R\rightarrow\R$ with $k=1$, $n=2$, defined as
 $$
 f(\xi,\eta):=(\ast\xi)\eta,\text{ for all }\xi\in\Lambda^2,\eta\in\R.
 $$
 While $f^\xi$, $f_\eta$ are affine for all $\xi\in\Lambda^{2}$ and $\eta\in\R$, $f$ is not ext-int.\ one convex. Theorem \ref{Thm principal ext-int. quasiaffine} and Corollary \ref{corollary to the main theorem} discuss how much of the converse of Lemma \ref{corollary of crucial lemma for extintaffine} is true in the category of ext-int.\ one affine functions.
\end{remark}
\begin{proof} To prove $(i)$, it is enough to see that for any $a\in\Lambda^1$, $b\in\Lambda^{k}$, there exist $c\in\Lambda^1$, $d\in\Lambda^{k}$ such that
$c\wedge d=a\wedge b$ and $c\lrcorner d=0$, which is a consequence of Lemma \ref{decomposition lemma}. One can prove $(ii)$ in the same spirit.
\end{proof}
\section{Some algebraic lemmas}\label{algebraiclemmas}
In this section, we prove few algebraic results required to prove the main theorem. The following lemma is elementary.
\begin{lemma}[Decomposition lemma]\label{decomposition lemma}
Let $1\leqslant k\leqslant n$, let $\omega\in\Lambda^k$ and let $x\in S^{n-1}$. Then, there exist $\omega_T(x)\in \Lambda^{k-1}(\{x\}^\perp)$ and $\omega_N(x)\in \Lambda^{k}(\{x\}^\perp)$
such that
$$
\omega=x\wedge \omega_T(x)+\omega_N (x).
$$
\end{lemma}
\begin{remark}
Note that $\omega_T(x)=x\lrcorner\omega,\,x\lrcorner\omega_T(x)=0\text{ and }x\lrcorner\omega_N(x)=0.$ In the sequel, we will write $\omega_T$ and $\omega_N$ instead of $\omega_T(e^1)$ and $\omega_N(e^1)$ respectively.
\end{remark}
The following function will have a recurrent appearance in the subsequent discussion.
\begin{definition}
Let $k,p,n\in\N$, $2\leqslant k\leqslant n$ and let us suppose that $\mathcal{D}^A\in \Lambda^{kp}$ satisfy $e^1\lrcorner \mathcal{D}^A=0$, for all $A\in\mathcal{T}^1_{k-1}$. We define $\mathcal{F}_p:\Lambda^k\times\Lambda^k\rightarrow\R$ as
{\allowdisplaybreaks
\begin{equation*}
\mathcal{F}_p(\omega,\alpha):=\sum_{A\in\mathcal{T}^1_{k-1}}\langle \mathcal{D}^A;\omega^{p-1}\wedge\alpha\rangle\langle\alpha;e^1\wedge e^A\rangle,\text{ for all }\omega,\alpha\in\Lambda^k.
\end{equation*}}
\end{definition}
The following lemma isolates the algebraic consequence of ext. one affinity.
\begin{lemma}\label{orthogonality 31.7.2012}
Let $k,n\in\N$ and let $2\leqslant k\leqslant n$. For all $J\in\mathcal{T}^1_{k-1}$, let $\mathcal{D}^J\in\Lambda^k$ satisfy
$e^1\lrcorner \mathcal{D}^J=0$ and let
\begin{equation}\label{20.8.2012.5}
\mathcal{F}_1(\omega,a\wedge b)=0,\text{ for all }\omega\in\Lambda^k,a\in\Lambda^1,b\in\Lambda^{k-1}.
\end{equation}
Then, for all $I,R\in \Tt{k-1}$ and $J,S\in \TT{k}$ satisfying $I\cap J=R\cap S=\emptyset$ and $I\cup J=R\cup S$, we have
\begin{equation}\label{21.9.2012.11052015}
\sgn{I,J}\left\langle\mathcal{D}^I;e^J\right\rangle=(-1)^k\sgn{R,S}\left\langle\mathcal{D}^R;e^S\right\rangle
\end{equation}
Hence, if either $k$ is odd or $2k>n$,
\begin{equation}\label{11.5.2015.1}
\mathcal{F}_1(\omega,\omega)=0,\text{ for all }\omega\in\Lambda^k.
\end{equation}
\end{lemma}
\begin{remark}
 As we will see later, forms $\mathcal{D}^J$ are connected to the coefficients of a ext.\ one affine function, which, as it will turn out, is a polynomial.
 In the proof of Theorem \ref{characterizations of affine functions}, we will see that Equation \eqref{20.8.2012.5} is basically the property of being affine in the direction of one-divisible forms in a different guise.
\end{remark}
\begin{proof}
We begin by noting that, for all $J\in\mathcal{T}^1_{k-1}$,
$e^J\lrcorner\mathcal{D}^J=0$.
Indeed, for a fixed $R\in\mathcal{T}^1_{k-1}$, it follows from Equation \eqref{20.8.2012.5} that,
$$
0= \sum_{J\in\mathcal{T}^1_{k-1}}\langle e^R\lrcorner \mathcal{D}^J;a\rangle \left\langle e^R\lrcorner\left(e^1\wedge e^J\right); a\right\rangle=-\langle e^R\lrcorner \mathcal{D}^R;a\rangle \left\langle e^1; a\right\rangle,\text{ for all }a\in\Lambda^1.
$$
This implies that $e^R\lrcorner\mathcal{D}^R=0$. Therefore, for all $R,S\in\mathcal{T}^1_{k-1}$ with $R\neq S$, on setting $b:=e^R+e^S$, it follows from Equation \eqref{20.8.2012.5} that
{\allowdisplaybreaks
\begin{align*}
0=&\sum_{J\in\mathcal{T}^1_{k-1}}\langle \left(e^R+e^S\right)\lrcorner \mathcal{D}^J;a\rangle \left\langle \left(e^R+e^S\right)\lrcorner\left(e^1\wedge e^J\right); a\right\rangle\\
=&\langle \left(e^R+e^S\right)\lrcorner \mathcal{D}^R;a\rangle \left\langle \left(e^R+e^S\right)\lrcorner\left(e^1\wedge e^R\right); a\right\rangle\\&+ \langle \left(e^R+e^S\right)\lrcorner \mathcal{D}^S;a\rangle \left\langle \left(e^R+e^S\right)\lrcorner\left(e^1\wedge e^S\right); a\right\rangle\\
=&-\langle e^S\lrcorner \mathcal{D}^R;a\rangle \left\langle e^1; a\right\rangle- \langle e^R\lrcorner \mathcal{D}^S;a\rangle \left\langle e^1; a\right\rangle=-\langle e^S\lrcorner \mathcal{D}^R+ e^R\lrcorner \mathcal{D}^S;a\rangle \left\langle e^1; a\right\rangle,
\end{align*}}
for all $a\in\Lambda^1$. Hence, we have proved that
\begin{equation}\label{21.9.2012.2}
e^R\lrcorner\mathcal{D}^S+ e^S\lrcorner\mathcal{D}^R=0,\text{ for all }R,S\in\mathcal{T}^1_{k-1}.
\end{equation}
We now claim that, for all $J\in\mathcal{T}^1_{k-1}$,
\begin{equation}\label{22.9.2012.1}
e^j\lrcorner\mathcal{D}^J=0,\text{ for all }j\in\{1\}\cup J.
\end{equation}
To see this, let $R\in\mathcal{T}^1_{k-1}$ and let $p\in\{1\}\cup R$ be fixed. To avoid the trivial case, let us assume that $p\in R$. It is enough to prove that
\begin{equation}\label{20.8.2012.7}
\langle e^p\lrcorner\mathcal{D}^R;e^S\rangle=0,\text{ for all }S\in\mathcal{T}_{k-1}.
\end{equation}
Let $S\in\mathcal{T}_{k-1}$ be given. If $1\in S$, it follows from the hypothesis that
$
\langle e^p\lrcorner\mathcal{D}^R;e^S\rangle=0.
$
Also, if $p\in S$, we deduce that
$
\langle e^p\lrcorner\mathcal{D}^R;e^S\rangle=\langle \mathcal{D}^R;e^p\wedge e^S\rangle=0.
$
Therefore, we can assume that $1,p\notin S$. Note that, $R\neq S$ because $p\in R$. It follows from Equation \eqref{21.9.2012.2} that, as $p\in R$,
$$
0=\left\langle e^S\lrcorner \mathcal{D}^R+ e^R\lrcorner \mathcal{D}^S;e^p\right\rangle
=\left\langle\mathcal{D}^R; e^p\wedge e^S\right\rangle=\left\langle e^p\lrcorner\mathcal{D}^R;e^S\right\rangle,
$$
which proves Equation \eqref{20.8.2012.7}. It remains to prove Equation \eqref{21.9.2012.11052015}. To avoid the trivial case, we assume that $1\notin J\cup S$. Let us now write
{\allowdisplaybreaks
\begin{align*}
I:=(i_1,\ldots,i_{k-1});&\,J:=(j_1,\ldots,j_{k}),\\
R:=(r_1,\ldots,r_{k-1});&\,S:=(s_1,\ldots,s_{k}).
\end{align*}
}
Note that, using Equation \eqref{21.9.2012.2}, we deduce that, for all $P,Q\in\Tt{k-1}$ and $r\in\{1,\ldots,n\}$,
\begin{equation}\label{21.9.2012.5}
\left\langle\mathcal{D}^P; e^Q\wedge e^r\right\rangle+\left\langle\mathcal{D}^Q; e^P\wedge e^r\right\rangle=0.
\end{equation}
We prove Equation \eqref{21.9.2012.11052015} by induction on $\card{I\cap R}$. First, let us prove Equation \eqref{21.9.2012.11052015} when $\card{I\cap R}=0$ i.e. $I\cap R=\emptyset$. In this case, for some $1\leqslant p,q\leqslant k$, we have  $I=\left(s_1,\ldots,s_{\hat{p}},\ldots,s_k\right)=S(s_p)$ and $R=\left(j_1,\ldots,j_{\hat{q}},\ldots,j_k\right)=J(j_q)$, with $s_p=j_q$. Therefore, it follows from Equation \eqref{21.9.2012.5} that
{\allowdisplaybreaks
\begin{align}\label{21.9.2012.6}
\left\langle\mathcal{D}^I; e^J\right\rangle=& (-1)^{k-q}\left\langle\mathcal{D}^{S(s_p)}; e^{J(j_q)}\wedge e^{j_q}\right\rangle
= (-1)^{k-q+1}\left\langle\mathcal{D}^{J(j_q)}; e^{S(s_p)}\wedge e^{j_q}\right\rangle\notag\\ =& (-1)^{k-q+1}\left\langle\mathcal{D}^{R}; e^{S(s_p)}\wedge e^{s_p}\right\rangle=(-1)^{p+q+1}\left\langle\mathcal{D}^{R}; e^{S}\right\rangle.
\end{align}
}
Furthermore, we observe that
\begin{equation}\label{21.9.2012.7}
\sgn{I,J}=(-1)^{q+1+k-p}\sgn{R,S}.
\end{equation}
Combining Equations \eqref{21.9.2012.6} and \eqref{21.9.2012.7}, Equation \eqref{21.9.2012.11052015} follows when $\card{I\cap R}=0$. Let us now assume that Equation \eqref{21.9.2012.11052015} holds true when $\card{I\cap R}=0,\ldots,p$, for some $p\in\{0,\ldots,k-1\}$. We prove Equation \eqref{21.9.2012.11052015} when $\card{I\cap R}=p+1$, where $p+1\leqslant k-1$. Since $J\setminus(I\cup R),\,S\setminus(I\cup R)\neq\emptyset$, let us choose $1\leqslant\mu\leqslant k-1$ and $1\leqslant\nu\leqslant k$ such that $i_\mu\in I\cap R$ and $j_\nu\in J\setminus(I\cup R)$. Clearly $i_\mu\neq j_\nu$. It follows from Equation \eqref{21.9.2012.5} that
{\allowdisplaybreaks
\begin{align}\label{21.9.2012.8}
\left\langle\mathcal{D}^I; e^J\right\rangle=& (-1)^{k-\nu}\left\langle\mathcal{D}^{I}; e^{J(j_\nu)}\wedge e^{j_{\nu}}\right\rangle
= (-1)^{k-\nu+1}\left\langle\mathcal{D}^{J(j_\nu)}; e^{I}\wedge e^{j_\nu}\right\rangle\notag\\
=&  (-1)^{\mu+\nu}\sgn{I(i_\mu),j_\nu}\sgn{J(j_\nu),i_\mu} \left\langle\mathcal{D}^{\left[I(i_\mu);{j_\nu}\right]}; e^{\left[J(j_\nu);{i_\mu}\right]}\right\rangle.
\end{align}
}
Since $$\left(I(i_\mu)\cup\{j_\nu\} \right)\cup \left(J(j_\nu)\cup\{i_\mu\} \right)=R\cup S, \left(I(i_\mu)\cup\{j_\nu\} \right)\cap \left(J(j_\nu)\cup\{i_\mu\} \right)=\emptyset,$$ and $\card{\left(I(i_\mu)\cup\{j_\nu\} \right)\cap R}=p$, it follows from the induction hypothesis that

\begin{equation}\label{21.9.2012.9}
\sgn{\left[I(i_\mu),{j_\nu}\right],\left[J(j_\nu);{i_\mu}\right]}\left\langle\mathcal{D}^{\left[I(i_\mu);{j_\nu}\right]}; e^{\left[J(j_\nu);{i_\mu}\right]}\right\rangle=(-1)^k\sgn{R,S}\left\langle\mathcal{D}^R; e^S\right\rangle.
\end{equation}
On noting that
$$
\sgn{\left[I(i_\mu),{j_\nu}\right],\left[J(j_\nu),{i_\mu}\right]}=(-1)^{\mu+\nu}\sgn{I,J} \sgn{I(i_\mu),{j_\nu}}\sgn{J(j_\nu),{i_\mu}},
$$
it follows from Equations \eqref{21.9.2012.8} and \eqref{21.9.2012.9} that
$$
\left\langle\mathcal{D}^I; e^J\right\rangle=(-1)^{k}\frac{\sgn{R,S}}{\sgn{I,J}}\left\langle\mathcal{D}^R; e^S\right\rangle,
$$
which proves the induction step. This proves Equation \eqref{21.9.2012.11052015}. To prove Equation \eqref{11.5.2015.1}, it is enough to prove that $\mathcal{D}^J=0$, for all $J\in\mathcal{T}^1_{k-1}$. If $k$ is odd, this follows from Equations \eqref{21.9.2012.11052015} and \eqref{22.9.2012.1}. When $2k>n$, let us assume to the contrary that $\mathcal{D}^J\neq 0$, for some $J\in\mathcal{T}^1_{k-1}$. Therefore, $\text{\upshape{rank}}_{1}\left(\mathcal{D}^J\right)\geqslant k$, see Proposition 2.37 of \cite{CDK}. Furthermore, using Equation \eqref{22.9.2012.1}, we deduce that
$$
\{e^1,e^r:r\in J\}\subseteq\text{\upshape{Ann}}_{\lrcorner}\left(\mathcal{D}^J,1\right).
$$
Therefore,
$$
k\leqslant \text{\upshape{rank}}_{1}\left(\mathcal{D}^J\right)=n-\operatorname*{dim}\left(\text{\upshape{Ann}}_{\lrcorner}\left(\mathcal{D}^J,1\right)\right)\leqslant n-|J|-1=n-k.
$$
This implies that $2k\leqslant n$ which is a contradiction. Hence $\mathcal{D}^J=0$. This proves the lemma.
\end{proof}
\begin{lemma}\label{23.9.2012.1}
Let $k,p,n\in\N$, $k\geqslant 2$, let $\mathcal{D}^A\in \Lambda^{kp}$ satisfy $e^1\lrcorner \mathcal{D}^A=0$, for all $A\in\mathcal{T}^1_{k-1}$, and let
\begin{equation}\label{23.9.2012.2}
\mathcal{F}_p(\omega,a\wedge b)=0,\text{ for all }\omega\in\Lambda^k,a\in\Lambda^1,b\in\Lambda^{k-1}.
\end{equation}
Then, for some $H_p\in\Lambda^{kp+k-1}$ with $e^1\lrcorner H_p=0$,
\begin{equation}\label{11.5.2015.2}
\mathcal{F}_p(\omega,\omega)=\langle e^1\wedge H_p;\omega^{p+1}\rangle,\text{ for all }\omega\in\Lambda^k.
\end{equation}
\end{lemma}
\begin{proof}
Let us begin by proving that if $k$ is even, for all $I,R\in \Tt{k-1}$ and $J,S\in \TT{kp}$ satisfying $I\cap J=R\cap S=\emptyset$ and $I\cup J=R\cup S$,
we have
\begin{equation}\label{23.9.2012.3}
\sgn{I,J}\left\langle\mathcal{D}^I;e^J\right\rangle=\sgn{R,S}\left\langle\mathcal{D}^R;e^S\right\rangle.
\end{equation}
The proof is very similar to that of Equation \eqref{21.9.2012.11052015} of Lemma \ref{orthogonality 31.7.2012}. To avoid the trivial case, let us assume that $kp\leqslant n$. If $p=1$, Equation \eqref{23.9.2012.3} follows from Lemma \ref{orthogonality 31.7.2012}. So, we assume $p\geqslant 2$. At the outset, let us observe that for all $Q\in\TT{(p-1)k}$, there exists $\omega\in\Lambda^k$ satisfying
\begin{equation}\label{23.9.2012.4}
\omega^{p-1}=e^Q.
\end{equation}
Indeed, for $Q:=\left(q_1,\ldots,q_{(p-1)k}\right)\in\TT{(p-1)k}$, the form $\omega\in\Lambda^k$ defined by
$$
\omega:=\frac{1}{(p-1)!}\sum_{r=0}^{p-2}e^{q_{rk+1}}\wedge\cdots\wedge e^{q_{(r+1)k}},
$$
satisfies Equation \eqref{23.9.2012.4}. Therefore, it follows from Equations \eqref{23.9.2012.2} and \eqref{23.9.2012.4} that, for all $a\in\Lambda^1$, $b\in\Lambda^{k-1}$ and $Q\in\TT{(p-1)k}$,
\begin{equation}\label{24.9.2012.2}
\sum_{A\in\mathcal{T}^1_{k-1}}\langle e^Q\lrcorner \mathcal{D}^A;a\wedge b\rangle\langle a\wedge b;e^1\wedge e^A\rangle=0.
\end{equation}
The rest of the proof of Equation \eqref{23.9.2012.3} follows essentially from Lemma \ref{orthogonality 31.7.2012} and its proof. Note that,
\begin{equation}\label{24.9.2012.1}
e^i\lrcorner\mathcal{D}^I=0,\text{ for all }i\in\{1\}\cup I,I\in\mathcal{T}^1_{k-1}
\end{equation}
It remains to prove Equation \eqref{11.5.2015.2}. To avoid the trivial case, we assume $kp\leqslant n$. When $k$ is odd, $\mathcal{F}_p$ is evidently zero on the diagonal when $p\geqslant 2$. Hence, one can take $H_p=0$ in this case. When $p=1$ and $k$ is odd, it follows from Lemma \ref{orthogonality 31.7.2012} that $\mathcal{F}_1$ is zero on the diagonal. Therefore, we can set $H_1=0$ in this case as well. Hence, it is enough to settle the lemma for the case when $k$ is even. To define $H_p\in\Lambda^{kp+k-1}$, using Equation \eqref{23.9.2012.3}, we note that, for all $R\in\mathcal{T}^1_{k-1},S\in\mathcal{T}^1_{kp}$ and $R\cap S=\emptyset$, there exists $\alpha_{R\cup S}\in\R$ such that
\begin{equation}\label{19.9.2012.1}
\left\langle \mathcal{D}^R;e^S\right\rangle=\alpha_{R\cup S}\operatorname*{sgn}(R,S).
\end{equation}
Let us now define $H_p\in \Lambda^{kp+k-1}$ by
$$
H_p:=\frac{1}{p+1}\sum_{I\in\mathcal{T}^1_{kp+k-1}}\alpha_I e^I.
$$
It follows from Equation \eqref{19.9.2012.1} that $H_p$ is well-defined. Note that, $e^1\lrcorner H_p=0$. Furthermore, for all $\omega\in\Lambda^k$, it follows from Equation \eqref{24.9.2012.1} that
{\allowdisplaybreaks
\begin{align}
\mathcal{F}_{p}(\omega,\omega)=& \sum_{R\in\mathcal{T}^1_{k-1}}\langle \mathcal{D}^R;\omega^{p}\rangle\langle\omega;e^1\wedge e^R\rangle\notag\\
=& \sum_{R\in\mathcal{T}^1_{k-1}}\left(\sum_{S\in\mathcal{T}^1_{kp},R\cap S=\emptyset}\left\langle \mathcal{D}^R;e^S\right\rangle\left\langle\omega^{p};e^S\right\rangle\right)\left\langle\omega; e^1\wedge e^R\right\rangle\notag\\
=& \sum_{R\in\mathcal{T}^1_{k-1}}\left(\sum_{S\in\mathcal{T}^1_{kp},R\cap S=\emptyset}\alpha_{R\cup S}\operatorname*{sgn}(R,S)\left\langle\omega^{p};e^S\right\rangle\right)\left\langle\omega; e^1\wedge e^R\right\rangle\notag\\
=& \sum_{I\in\mathcal{T}^1_{k-1+kp}}\alpha_{I}\left(\sum_{\substack{R\in\mathcal{T}^1_{k-1},S\in\mathcal{T}^1_{kp},\\R\cup S=I,R\cap S=\emptyset}}\operatorname*{sgn}(R,S)\left\langle\left(\omega^{p}\right)_N;e^S\right\rangle\left\langle\omega_T; e^R\right\rangle\right)\notag\\
=& \sum_{I\in\mathcal{T}^1_{k-1+kp}}\alpha_{I}\left\langle \omega_T\wedge \left(\omega^{p}\right)_N; e^I\right\rangle = (p+1)\left\langle H_p; \omega_T\wedge \left(\omega^{p}\right)_N\right\rangle .\notag
\end{align}
}
Since $e^1\lrcorner H_p=0$ and $k$ is even, it follows that,
$$
\mathcal{F}_p(\omega,\omega)= (p+1)\left\langle H_p; \omega_T\wedge \left(\omega^{p}\right)_N\right\rangle=\left\langle e^1\wedge H_p;\omega^{p+1}\right\rangle, \text{ for all }\omega\in\Lambda^k,
$$
which proves the lemma.
\end{proof}
\section{Characterization of ext-int.\ one affine functions}\label{main}
Let us begin by characterizing ext.\ one affine functions.
\begin{theorem}\label{characterizations of affine functions}
Let $1\leqslant k\leqslant n$ and let $f:\Lambda^k\rightarrow\R$. Then, $f$ is ext.\ one affine if and only if there exist $m\in\N$ with $m\leqslant n$, $a_r\in\Lambda^{kr}$, where $r=0,\ldots,m$ such that
\begin{equation}\label{24.9.2012.4}
f(\omega)=\sum_{r=0}^{m}\langle a_r;\omega^r \rangle,\text{ for all }\omega\in\Lambda^k.
\end{equation}
\end{theorem}
\begin{remark} Note that, since $\omega^r=0$ for all $r>\left[\frac{n}{k}\right]$, it follows that $m\leqslant \left[\frac{n}{k}\right]$.
\end{remark}
\begin{proof} We show that any ext.\ one affine function $f:\Lambda^k\rightarrow\R$ is of the form \eqref{24.9.2012.4}. The converse part is easy to check. In view of Remark \ref{23.4.2014.1}, let us assume $k\geqslant 2$. The proof proceeds by induction on the dimension $n$. When $n=k$, the result follows easily. Let us assume that the theorem holds true when $n=k,\ldots,p$, for some $p\geqslant k$. We prove the result for $n=p+1$. It is given that $f:\Lambda^k\left(\R^{p+1}\right)\rightarrow\R$ is  ext.\ one affine.
\\ \\
Since $f$ is ext.\ one affine,
{\allowdisplaybreaks
\begin{align}\label{11.9.2012.1}
f(\omega)=&f\left(\omega_N\right)+\sum_{J\in\mathcal{T}^1_{k-1}}\omega_{1,J}\left( f\left(\omega_N+e^1\wedge e^J\right)-f\left(\omega_N\right)\right)\notag\\
=&f\left(\omega_N\right)+\sum_{J\in\mathcal{T}^1_{k-1}}\omega_{1,J}\left( f_{e^1\wedge e^J}\left(\omega_N\right)-f\left(\omega_N\right)\right),\text{ for all }\omega\in \Lambda^k,
\end{align}}
where, for all $J\in\mathcal{T}^1_{k-1}$, $f_{e^1\wedge e^J}:\Lambda^{k}\left(\{e^1\}^\perp\right)\rightarrow\R$ is defined as
$$
f_{e^1\wedge e^J}(\xi):=f\left(\xi+e^1\wedge e^J\right),\text{ for all }\xi\in \Lambda^{k}\left(\{e^1\}^\perp\right).
$$
Since $f:\Lambda^k\left(\R^{p+1}\right)\rightarrow\R$ is ext.\ one affine,  so are $f|_{\Lambda^{k}\left(\{e^1\}^\perp\right)}$ and $\,f_{e^1\wedge e^J}$, for all $J\in\mathcal{T}^1_{k-1}$. Therefore, the  induction hypothesis ensures the existence of  $m(p,k)\in\N$, $m(p,k)\leqslant p$ and $a^0_0,a^J_0\in\R$, $a^0_r,a^J_r\in\Lambda^{kr}\left(\{e^1\}^\perp\right)$ for all $J\in\mathcal{T}^1_{k-1}$ and $r=1,\ldots,m(p,k)$, satisfying
\begin{equation*}\label{12.9.2012.1}
f(\varphi)=a^0_0+\sum_{r=1}^{m(p,k)}\langle a^0_r;\varphi^r\rangle,\text{ for all }\varphi \in \Lambda^k\left(\{e^1\}^\perp\right),
\end{equation*}
and
\begin{equation*}\label{12.9.2012.2}
f_{e^1\wedge e^J}(\varphi)=a^J_0+\sum_{r=1}^{m(p,k)}\langle a^J_r;\varphi^r\rangle,\text{ for all }\varphi \in \Lambda^k\left(\{e^1\}^\perp\right).
\end{equation*}
Thus, it follows from Equation \eqref{11.9.2012.1} that, for all $\omega\in\Lambda^k$,
{\allowdisplaybreaks
\begin{align}\label{12.9.2012.3}
f(\omega)=& \left(a^0_0+\sum_{r=1}^{m(p,k)}\langle a^0_r;\omega_N^r\rangle \right) +\sum_{J\in\mathcal{T}^1_{k-1}}\omega_{1,J}\left( a^J_0-a^0_0+ \sum_{r=1}^{m(p,k)}\langle a^J_r-a^0_r;\omega_N^r\rangle\right)\notag\\
=& \left(a^0_0+\sum_{r=1}^{m(p,k)}\langle\overline{a}^0_r;\omega^r\rangle \right) + \sum_{J\in\mathcal{T}^1_{k-1}}\sum_{r=1}^{m(p,k)}\langle \mathcal{D}^J_r;\omega^r\rangle\left\langle\omega;e^1\wedge e^J\right\rangle\notag\\
=& \left(a^0_0+\sum_{r=1}^{m(p,k)}\langle\overline{a}^0_r;\omega^r\rangle \right) + \sum_{r=1}^{m(p,k)}\mathcal{F}_r(\omega,\omega),
\end{align}}
where
$$
\overline{a}^0_r:=\left\{\begin{array}{rl}
a^0_1+e^1\wedge\left(\sum_{J\in\mathcal{T}^1_{k-1}}\left(a^J_0-a^0_0\right)e^J\right),&\text{ if }r=1,\\
a^0_r,&\text{ if }r=2,\ldots,m(p,k).
\end{array}\right.
$$
and
$$
\mathcal{D}^J_r:=a^J_r-a^0_r,\text{ for all }J\in\mathcal{T}^1_{k-1}\text{ and }r=1,\ldots,m(p,k).
$$
Note that, for all $J\in\mathcal{T}^1_{k-1}$, and $r=1,\ldots,m(p,k)$, $e^1\lrcorner\mathcal{D}^J_{r}=0$. Since $f$ is ext.\ one affine,
\begin{equation*}
\sum_{r=1}^{m(p,k)}r\mathcal{F}_r(\omega,c\wedge d)=0,\text{ for all }\omega\in\Lambda^k,c\in\Lambda^{1},d\in\Lambda^{k-1}.
\end{equation*}
Hence, by different degree of homogeneity, for all $r=1,\ldots,m(p,k)$,
$$
\mathcal{F}_r(\omega,c\wedge d)=0,\text{ for all }\omega\in\Lambda^{k},c\in\Lambda^1,d\in\Lambda^{k-1}.
$$
We invoke Lemma \ref{23.9.2012.1} at this point to find $G_r\in\Lambda^{kr+k}$, for all $r=1,\ldots,m(p,k)$,
\begin{equation*}
\mathcal{F}_r(\omega,\omega)=\langle G_r;\omega^{r+1}\rangle,\text{ for all }\omega\in\Lambda^k,
\end{equation*}
from where the result follows using Equation \eqref{12.9.2012.3}. This completes the proof.
\end{proof}
\\ \\
Invoking the Hodge transform, see Remark \ref{23.4.2014.1}, the characterization of int.\ one affine functions follows immediately from Theorem \ref{characterizations of affine functions}.
\begin{corollary}\label{characterizations of intaffine functions}
Let $0\leqslant k\leqslant n-1$ and let $f:\Lambda^k\rightarrow\R$. Then, $f$ is int.\ one affine if and only if there exists $a_r\in\Lambda^{(n-k)r}$, for all $r=0,\ldots,\left[\frac{n}{n-k}\right]$, such that
\begin{equation}\label{24.9.2015.4}
f(\omega)=\sum_{r=0}^{\left[\frac{n}{n-k}\right]}\langle a_r;(\ast\omega)^r \rangle,\text{ for all }\omega\in\Lambda^k.
\end{equation}
\end{corollary}
An interesting consequence of Theorem \ref{characterizations of affine functions} and Corollary \ref{characterizations of intaffine functions} is the following.
\begin{theorem}\label{5.9.2013.1}
Let $1\leqslant k\leqslant n-1$ satisfy $2k\neq n$. Then, $f:\Lambda^k\rightarrow\mathbb{\R}$ is affine if and only if $f$ is both ext. and int.\ one affine.
\end{theorem}
\begin{remark}Theorem \ref{5.9.2013.1} does not hold if $2k=n$ with $k$ even. To see this, define $f:\Lambda^k\left(\mathbb{R}^{2k}\right)\rightarrow\mathbb{R}$ by
$$
f(\omega):=\left\langle e^1\wedge\cdots e^{2k};\omega\wedge\omega\right\rangle,\text{ for all }\omega\in \Lambda^k\left(\mathbb{R}^{2k}\right).
$$
\end{remark}
\begin{proof} If $2k>n$, the conclusion follows trivially from Theorem \ref{characterizations of affine functions}. If $2k<n$, i.e. $n<2(n-k)$, since $f$ is int.\ one affine, using Corollary \ref{characterizations of intaffine functions}, we deduce that $f$ is affine.
\end{proof}
\begin{theorem}[Characterization of ext-int.\ one affine functions]
\label{Thm principal ext-int. quasiaffine}Let $1\leqslant k\leqslant n-1$ and $f:\Lambda^{k+1} \times \Lambda^{k-1} \rightarrow\mathbb{R}.$ The following statements are equivalent
\begin{itemize}
\item [$(i)$] $f$ is ext-int.\ polyaffine.
\item [$(ii)$] $f$ is ext-int.\ quasiaffine.
\item [$(iii)$] $f$ is ext-int.\ one affine.
\item [$(iv)$] For all $0\leqslant s\leqslant\left[
\frac{n}{k+1}\right]$ and $0 \leqslant r\leqslant\left[\frac{n}{n-k+1}\right] $, there exist $c_{s}\in\Lambda^{(k+1)s}$, $d_{r}\in\Lambda^{(n-k+1)r}$ such that
\[
f\left(  \xi , \eta \right)  =\sum_{s=0}^{\left[ \frac{n}{k+1}\right]  }\left\langle c_{s}%
;\xi^{s}\right\rangle  + \sum_{r=0}^{\left[  \frac{n}{n-k+1} \right]  }\left\langle d_{r}%
;(\ast\eta)^{r}\right\rangle,\text{ for all }\xi\in\Lambda^{k+1},\eta\in\Lambda^{k-1}.
\]
\end{itemize}
\end{theorem}
Theorem \ref{Thm principal ext-int. quasiaffine} has the curious implication that nonlinearity can trickle into an ext-int.\ one affine function at the most through one variable. This is formally stated in the following corollary whose proof is easy enough to skip over.
\begin{corollary}\label{corollary to the main theorem}
Let $1\leqslant k\leqslant n-1$. Then $f:\Lambda^{k+1} \times \Lambda^{k-1} \rightarrow\mathbb{R}$ is ext-int.\ one affine if and only if there exist $g:\Lambda^{k+1}\rightarrow\R$ and $h:\Lambda^{k-1}\rightarrow\R$ such that
$$
f\left(  \xi , \eta \right)  =g\left(\xi\right)+h\left(\eta \right),\text{ for all }\xi\in\Lambda^{k+1},\eta\in\Lambda^{k-1},
$$
where
\begin{itemize}
\item [$(i)$] $g$ is affine and $h$ is int.\ one affine, when $n\leqslant 2k-2$.
\item [$(ii)$]$g$, $h$ are affine, when $n=2k-1,2k,2k+1$, or both of $k,n$ are even.
\item [$(iii)$]$g$ is ext.\ one affine and $h$ is affine, when $n\geqslant 2k+2$.
\end{itemize}
\end{corollary}
\begin{remark} There is no analogue of Corollary \ref{corollary to the main theorem} at the level of ext-int.\ polyconvexity. In other words, there are ext-int.\ polyconvex functions that cannot be written as a sum of ext.\ polyconvex and int.\ polyconvex functions. To see this, one may consider the following function $f:\Lambda^2\times\R\rightarrow\R$, with $k=1,n=4$, defined as
$$
f\left(\xi,\eta\right):=e^{\lvert\xi\wedge\xi\rvert^2+\eta^2},\text{ for all }\xi\in\Lambda^2,\eta\in\R.
$$
\end{remark}
Let us proceed to proving Theorem \ref{Thm principal ext-int. quasiaffine}.\\ \\
\begin{proof} The chain of implications $(i) \Rightarrow (ii) \Rightarrow (iii) $, follows from standard techniques of classical calculus of variations, see \cite{Sil} for details. It is obvious from the definition of ext-int. polyconvexity that $(iv)$ implies $(i)$. It remains to prove $(iii) \Rightarrow (iv)$. Let us divide the proof in four steps.
\\ \\
\noindent\emph{ Step 1:} For each $\eta\in\Lambda^{k-1}$, we use Lemma \ref{corollary of crucial lemma for extintaffine} and Theorem \ref{characterizations of affine functions} to find $c_{s} (\eta ) \in \Lambda^{(k+1)s}$ for all $s=0,\ldots,\left[ \frac{n}{k+1}\right]$, such that
\begin{equation}
 f\left(  \xi , \eta \right)  = f_{\eta}\left(\xi\right)=\sum_{s=0}^{\left[ \frac{n}{k+1}\right]  }\left\langle c_{s} (\eta )
;\xi^{s}\right\rangle ,\text{ for all }\xi\in\Lambda^{k+1},\eta\in \Lambda^{k-1},
\end{equation}
where $c_{s}:\Lambda^{k-1}\rightarrow\Lambda^{(k+1)s}$ is such that the function $\eta \mapsto f(\xi, \eta)$ is int.\ one affine for every $\xi \in \Lambda^{k+1}$. Defining $$f_{s}(\xi, \eta):= \left\langle c_{s} (\eta ) ;\xi^{s}\right\rangle,\text{ for all }\xi\in\Lambda^{k+1},\eta\in \Lambda^{k-1},$$ we see that due to
different degrees of homogeneity in $\xi$, for each $s$, $f_{s}$ is ext-int.\ one affine. Hence, it is enough to consider each $f_s$ separately, with $0 \leqslant s \leqslant \left[ \frac{n}{k+1}\right]$.
\\ \\
\emph{Step 2:} Let $0 \leqslant s \leqslant \left[ \frac{n}{k+1}\right]$ be fixed. Let us write
\begin{equation}\label{29.4.2015.1}
 f_{s}(\xi, \eta) = \sum_{I \in \mathcal{T}_{(k+1)s}} c_{s}^{I} (\eta )(\xi^{s})_{I},
\end{equation}
where $c_{s}^{I} (\eta )$, $(\xi^{s})_{I}$ denote the $I$-th component of $c_{s}(\eta )$ and $\xi^{s}$ respectively, for all $I \in \mathcal{T}_{(k+1)s}$. We claim that for each
multiindex $I \in \mathcal{T}_{(k+1)s}$, $c_{s}^{I}$ is int.\ one affine.
\\ \\
Indeed, there is nothing to prove when $s=0$. When
$1 \leqslant s \leqslant \left[ \frac{n}{k+1}\right]$ and $I = \left(i_1,\ldots, i_{(k+1)s}\right)\in \mathcal{T}_{(k+1)s}$, on setting
$$\xi_1 := e^{i_1}\wedge\cdots\wedge e^{i_{k+1}} +  e^{i_{k+2}}\wedge\cdots\wedge e^{i_{2(k+1)}} + \cdots +
e^{i_{(k+1)(s-1) +1}}\wedge \cdots\wedge e^{i_{(k+1)s}},$$
we see that
 $$f_{s}(\xi_1, \eta) = s!\,c_{s}^{I} (\eta ),\text{ for all }\eta\in\Lambda^{k-1},$$
from where it follows that $c_{s}^{I}$ is int.\ one affine as $f_{s}$ is ext-int.\ one affine, see Lemma \ref{corollary of crucial lemma for extintaffine}. This proves the claim. \smallskip
\\ \\
\emph{Step 3:} Invoking Corollary \ref{characterizations of intaffine functions}, it follows from Step 2 that
\begin{equation*}
 c_{s}^{I}(\eta) = \sum_{r=0}^{\left[  \frac{n}{n-k+1} \right]  }\left\langle d^{I}_{r,s};(\ast\eta)^{r}\right\rangle,\text{ for all }\eta\in\Lambda^{k-1}.
\end{equation*}
Therefore, using Equation \eqref{29.4.2015.1},
\begin{align}
f_{s}(\xi, \eta) &= \sum_{I \in \mathcal{T}_{(k+1)s}}\left(\sum_{r=0}^{\left[  \frac{n}{n-k+1} \right]  }\left\langle
d^{I}_{r,s};(\ast\eta)^{r}\right\rangle \right)(\xi^{s})_{I}\notag\\
  &= \sum_{r=0}^{\left[  \frac{n}{n-k+1} \right]}  \left(  \sum_{I \in \mathcal{T}_{(k+1)s}}\left\langle d^{I}_{r,s};(\ast\eta)^{r}\right\rangle(\xi^{s})_{I}\right),\text{ for all }\xi\in\Lambda^{k+1},\eta\in \Lambda^{k-1}.
\end{align}
Once again, by different degree of homogeneity in $\ast \eta$, it is enough to consider fixed but arbitrary $r$ with $ 0 \leqslant r \leqslant \left[  \frac{n}{n-k+1}\right] $. To that effect, we define $f_{r,s}:\Lambda^{k+1}\times\Lambda^{k-1}\rightarrow\R$ as
$$ f_{r,s}( \xi, \eta) :=  \sum_{I \in \mathcal{T}_{(k+1)s}}\left\langle d^{I}_{r,s};(\ast\eta)^{r}\right\rangle(\xi^{s})_{I},\text{ for all }\xi\in\Lambda^{k+1},\eta\in \Lambda^{k-1}.$$ This can be written as,
\begin{equation}\label{rsextintterm}
 f_{r,s}(\xi, \eta ) = \sum_{I \in \mathcal{T}_{(k+1)s}} \sum_{J \in \mathcal{T}_{(n-k+1)r}} d^{I,J}_{r,s}((\ast\eta)^{r})_{J}(\xi^{s})_{I}
\end{equation}
\noindent\emph{Step 4:} \smallskip To finish the proof, it is enough to prove that for all $I \in \mathcal{T}_{(k+1)s}$ and $J \in \mathcal{T}_{(n-k+1)r}$,
\begin{equation}\label{30.4.2015.3}
d^{I,J}_{r,s} = 0,\text{ for all } 1 \leqslant s \leqslant \left[ \frac{n}{k+1}\right],\,1 \leqslant r \leqslant \left[  \frac{n}{n-k+1}\right].
\end{equation}
We now proceed to show that. Let $ 1 \leqslant s \leqslant \left[ \frac{n}{k+1}\right]$, $1 \leqslant r \leqslant \left[  \frac{n}{n-k+1}\right]$ be fixed. Note that, for any $I \in \mathcal{T}_{(k+1)s}$ and
$J \in \mathcal{T}_{(n-k+1)r} $, $I\cap J$ has at least one element (In fact, there must be at least two).
Let us write $I = \left( i_1,\ldots, i_{(k+1)s}\right)\in \mathcal{T}_{(k+1)s}$ and $J =\left( j_1,\ldots, j_{(n-k+1)r}\right)\in \mathcal{T}_{(n-k+1)r}$ with $i_p = j_q$ for some $p,q$.
\\ \\
Let us divide $I$ into $s$ blocks of multiindices $ I^{\alpha}$ each containing $k+1$ indices, where
 $$I^{\alpha}: = \left(i_{(\alpha - 1)(k+1) +1},
\ldots, i_{\alpha(k+1)}    \right),\text{ for all }1 \leqslant \alpha \leqslant s.$$
Similarly, we divide the multiindex $J$ into $r$ blocks of multiindices $J^{\beta}$,
each containing $n-k+1$ indices, $J^{\beta}$, where
$$J^{\beta}:= \left( j_{(\beta - 1)(n-k+1) +1},
\ldots, j_{\beta(n-k+1)}    \right)\text{ for all }1 \leqslant \beta \leqslant r.
$$
Furthermore, for the sake of clarity, let $I_{p} \in \mathcal{T}_{k+1}$ denote the block of $k+1$ indices of $I$  containing $i_p$  and $J_{q} \in \mathcal{T}_{n-k+1}$ denote the block of $n-k+1$ indices of $J$  which contains $j_q$. Note that, in our notation, this implies that
$$I_{p} = I^{\left[ \frac{p-1}{k+1} \right] +1 },\text{ and }
J_{q} = J^{\left[ \frac{q-1}{n-k+1} \right] +1 } .$$
Let us choose
\begin{equation}\label{30.4.2015.2}
 \left\lbrace
\begin{aligned}
 a &:= e^{i_p} = e^{j_q}, & b &:= e^{I_p(i_p)} + \ast\, e^{J_q(j_q)} , \\
\xi &:= \frac{1}{(s-1)!}\sum_{\substack{1 \leqslant \alpha \leqslant s \\ \alpha \neq \left[ \frac{p-1}{k+1} \right] +1 }} e^{ I^{ \alpha } } ,&
\ast \eta &:= \frac{1}{(r-1)!}\sum_{\substack{1 \leqslant \beta \leqslant r\\ \beta \neq \left[ \frac{q-1}{n-k+1} \right] +1 }} e^{I^{\beta}}.
\end{aligned}
\right.
\end{equation}
Of course, if $s=1$, we choose $\xi=0$, and if $r=1$, we choose $\ast \eta = 0$.
Clearly
\begin{align*}
 a \wedge b &= e^{i_p}\wedge e^{I_{p}\left(i_p\right)} +  e^{j_q}\wedge\left( \ast e^{J_{q}(j_q)} \right) = \operatorname*{sgn}\left(i_p,I_{p}(i_p)\right)e^{I_{p}},\\
a \wedge  \ast b  &= e^{i_p}\wedge \left(  \ast e^{I_{p}\left(i_p\right)} \right) +(-1)^{k(n-k)} e^{j_q}  \wedge  e^{J_{q}(j_q)}  \notag \\
   &=\operatorname*{sgn}\left(j_q,J_{q}(j_q)\right)(-1)^{k(n-k)} e^{J_{q}}.
\end{align*}
Moreover, we observe that
\begin{equation}\label{30.4.2015.1}
\xi ^{s-1} = e^{I\setminus I_{p}},\text{ and }
 \left( \ast \eta \right)^{r-1} = e^{J\setminus J_{q}} .
\end{equation}
Note that here we implicitly used the following facts. If $s = 1 \textrm{ or } 2$, Equation \eqref{30.4.2015.1} is
trivially true, and if $s \geqslant 2$, it follows that $k+1$ is even, for otherwise, terms containing $\xi^s$ are absent from the expression for $f$.
If $k+1$ is even, Equation \eqref{30.4.2015.1} is easily seen to hold for any $ 2 \leqslant s \leqslant \left[ \frac{n}{k+1}\right]$. One can similarly argue for $\ast\eta$.
\\ \\
Henceforth, we will disregard questions of signs, as it is not important for the argument and use $\pm$ to denote that either sign is possible. Using Equation \ref{rsextintterm}, we have, for any $t \in [ 0,1 ]$,
\begin{align*}
 f_{r,s} & (\xi + t a \wedge b , \eta + t a \lrcorner b)\\
  &= \sum_{K \in \mathcal{T}_{(k+1)s},L \in \mathcal{T}_{(n-k+1)r}} d^{K,L}_{r,s}((\ast \left( \eta +  t a \lrcorner b \right))^{r})_{L}
 (\left( \xi + t a \wedge b \right)^{s})_{K} \notag \\
  &=\sum_{K \in \mathcal{T}_{(k+1)s},L \in \mathcal{T}_{(n-k+1)r}} d^{K,L}_{r,s}((\ast \eta \pm  t a \wedge \left( \ast b \right))^{r})_{L}
 (\left( \xi + t a \wedge b \right)^{s})_{K}.
 \end{align*}
 With our choice of $a, b, \xi, \eta$ in Equation \eqref{30.4.2015.2},
the quadratic term in $t$, say $Q(t)$,  in the above expression on the right hand side is, for all $t\in [0,1]$,
 \begin{align*}
  Q(t)&=\pm t^{2}r!s! \sum_{K \in \mathcal{T}_{(k+1)s},L \in \mathcal{T}_{(n-k+1)r}} d^{K,L}_{r,s}
  \left( \left(\ast\eta \right)^{r-1}\wedge a \wedge \left(\ast b  \right)\right)_{L} \left( \xi^{s-1}\wedge a \wedge b    \right)_{K}\\
  &= \pm t^{2}r!s! \sum_{K \in \mathcal{T}_{(k+1)s},L \in \mathcal{T}_{(n-k+1)r}} d^{K,L}_{r,s}
  \left( e^{J\setminus J_{q}} \wedge \left( \pm e^{J_{q}}\right)\right)_{L} \left( e^{I\setminus I_{p}} \wedge \left( \pm e^{I_{p}} \right)    \right)_{K}\\
  &= \pm t^{2}r!s! \sum_{K \in \mathcal{T}_{(k+1)s},L \in \mathcal{T}_{(n-k+1)r}} d^{K,L}_{r,s}
  \left( \pm e^{J}\right)_{L} \left(  \pm e^{I} \right)_{K}
  = \pm t^{2}r!s! d^{I,J}_{r,s}.
 \end{align*}
Since $f_{r,s}$ is ext-int.\ one affine, $Q(t)=0$, for all $t\in [0,1]$, which forces
 $d^{I,J}_{r,s} = 0$. This proves Equation \eqref{30.4.2015.3} and the proof is complete. \end{proof}
 \\ \\
\textbf{Acknowledgement.} We have benefitted of interesting discussions with Professor Bernard Dacorogna. Part of this work was completed during visits of S. Bandyopadhyay to EPFL, whose hospitality and support is gratefully acknowledged. The research of S. Bandyopadhyay was partially supported by a SERB research project titled ``Pullback Equation for Differential Forms".

\end{document}